\documentclass{birkjour}

\usepackage[utf8]{inputenc}
\usepackage[T1]{fontenc}
\usepackage[english]{babel}
\usepackage{mathtools,amsmath,amssymb,amsfonts,amsthm,mathrsfs,latexsym}

\mathtoolsset{showonlyrefs}
\usepackage[hidelinks, pdftex]{hyperref}

\usepackage{graphicx,color,xcolor}
\usepackage{tikz,pgfplots,pgf}
\usepackage[caption=false]{epsfig}

\usepackage{geometry}
\usepackage{enumitem}

\newcommand{\R}{{\mathbb R}}

\def\HH{{\mathcal H}}
\def\UU{{\mathcal U}}
\def\PP{{\mathcal P}}
\def\II{{\mathcal I}}
\def\JJ{{\mathcal J}}
\def\AA{{\mathcal A}}
\def\LL{{\mathcal L}}
\def\CC{{\mathcal C}}
\def\EE{{\mathcal E}}
\def\TT{{\mathcal T}}
\def\YY{{\mathcal Y}}
\def\SuccMap{{\mathcal S}}
\DeclareMathOperator{\Rot}{Rot}

\DeclareFontFamily{U}{mathx}{}
\DeclareFontShape{U}{mathx}{m}{n}{<-> mathx10}{}
\DeclareSymbolFont{mathx}{U}{mathx}{m}{n}
\DeclareMathAccent{\bigcheck}{0}{mathx}{"71}
\DeclareMathAccent{\bighat}{0}{mathx}{"70}

\usepackage{lipsum}

\begingroup
\newtheorem{theorem}{Theorem}[section]
\newtheorem*{theorem*}{Theorem}
\newtheorem{lemma}[theorem]{Lemma}
\newtheorem{proposition}[theorem]{Proposition}
\newtheorem{corollary}[theorem]{Corollary}
\theoremstyle{definition}
\newtheorem{remark}[theorem]{Remark}
\endgroup

\definecolor{interno}{rgb}{0.85, 0.85, 0.85}
\definecolor{palla}{rgb}{0.7,0.7,0.7}
\definecolor{FFS-1}{rgb}{0.85, 0.85, 0.85}
\definecolor{FFS-2}{rgb}{0.7,0.7,0.7}

\numberwithin{equation}{section}

\begin{document}

\title[A Poincar\'e--Birkhoff theorem for multivalued successor maps]{A Poincar\'e--Birkhoff theorem \\for multivalued successor maps with applications \\to periodic superlinear Hamiltonian systems}

\author[G.~Feltrin]{Guglielmo Feltrin}

\address{
Dipartimento di Scienze Matematiche, Informatiche e Fisiche\\
Universit\`a degli Studi di Udine\\
Via delle Scienze 206, 33100 Udine, Italy}

\email{guglielmo.feltrin@uniud.it}

\author[A.~Fonda]{Alessandro Fonda}

\address{
Dipartimento di Matematica, Informatica e Geoscienze\\
Universit\`a degli Studi di Trieste\\
P.le Europa 1, 34127 Trieste, Italy}

\email{a.fonda@units.it}

\author[A.~Sfecci]{Andrea Sfecci}

\address{
Dipartimento di Matematica, Informatica e Geoscienze\\
Universit\`a degli Studi di Trieste\\
P.le Europa 1, 34127 Trieste, Italy}

\email{asfecci@units.it}



\date{}

\begin{abstract}
We provide a new version of the Poincar\'e--Birkhoff theorem for possibly multivalued successor maps associated with planar non-autonomous Hamiltonian systems. As an application, we prove the existence of periodic and subharmonic solutions of the scalar second order equation $\ddot x + \lambda g(t,x) = 0$, for $\lambda>0$ sufficiently small, with $g(t,x)$ having a superlinear growth at infinity, without requiring the existence of an equilibrium point.
\end{abstract}

\maketitle

\section{Introduction}\label{section-1}

In the search for periodic solutions to a differential system of ordinary differential equations, a classical approach is to look for the existence of fixed points of the associated \textit{Poincar\'e map}. 
It is worth noting, however, that there is a less known alternative strategy for planar systems, i.e., searching for fixed points of the so-called \textit{successor map}. This map already appears in \cite{Al-68}, even if the main idea may be traced back to Poincar\'e himself \cite{Po-BOOK}.

Consider, for instance, the planar system 
\begin{equation}\label{syst-fg}
\begin{cases}
\, \phantom{-}\dot x=f(t,x,y),
\\
\, -\dot y=g(t,x,y),
\end{cases}
\end{equation}
where $f,g \colon\mathbb{R}\times\mathbb{R}^2\to\mathbb{R}$ are continuous functions.
Assuming for the moment the uniqueness of the solution for initial value problems, we recall that, given $T>0$, the Poincar\'e map $\mathcal P_T \colon \R^2 \to \R^2$ associates to every $z_0=(x_0,y_0)$ the point $z(T)$, where $z=(x,y)$ is the solution of \eqref{syst-fg} such that $z(0)=z_0$.

In order to introduce the successor map let us first recall the definition of \textit{rotation number}.
Given a solution $z=(x,y)$ of \eqref{syst-fg}, defined on an interval $\mathopen{[}\alpha,\beta\mathclose{]}$ and satisfying $z(t)\neq0$ for every $t\in\mathopen{[}\alpha,\beta\mathclose{]}$, we can 
introduce its modified polar coordinates
\begin{equation}\label{modpolcor}
x= \sqrt{2\rho} \sin \theta, \qquad
y= \sqrt{2\rho} \cos \theta,
\end{equation}
and set
\begin{equation*}
\Rot(z,[\alpha,\beta])= \dfrac{1}{2\pi}\big(\theta(\beta)-\theta(\alpha)\big).
\end{equation*}

Let us now assume that every solution $z=(x,y)$ of \eqref{syst-fg}, with initial conditions
\begin{equation*}
x(t_0)=0, \quad y(t_0)=y_0>0,
\end{equation*}
is unique and has the following properties:
there exists $t_1>t_0$ such that $z$ is defined on the whole interval $\mathopen{[}t_0,t_1\mathclose{]}$ with $z(t)\neq0$ for every $t\in\mathopen{[}t_0,t_1\mathclose{]}$,
$z(t_1)=(0,y_1)$ with $y_1>0$,
and $\Rot(z,[t_0,t_1])=1$.
The successor map is defined as 
\begin{equation}\label{SM-def}
\SuccMap \colon \R\times\mathopen{]}0,+\infty\mathclose{[}\to \mathbb{R}\times\mathopen{]}0,+\infty\mathclose{[}, 
\qquad \SuccMap(t_0,y_0) =(t_1,y_1).
\end{equation}
Hence, denoting the components of $\mathcal S$ by $(\mathcal{T},\mathcal{Y})$, we have
\begin{equation}\label{SY}
\mathcal{T}(t_0,y_0)=t_1,
\quad
\mathcal{Y}(t_0,y_0)=y_1.
\end{equation}

The successor map can thus be seen as a kind of {\em first return map} on the half-plane $\{(t,x,y) \colon t\in \R, \, x=0, \, y>0\}$. Notice that the Poincar\'e map could also be seen as a first return map in the case of $T$-periodic systems after identifying the sections $\{(t,x,y)\colon t=0\}$ and $\{(t,x,y)\colon t=T\}$.

It could happen, in some situations, that the successor map is well-defined, while the Poincar\'e map is not. A typical example arises when dealing with scalar second order differential equations with a nonlinearity having a superlinear growth, when the global existence of the solutions is not guaranteed (see, e.g., \cite{CoUl-67}). Variants of the above described situation have been considered, e.g., in equations with singularities or when dealing with bouncing solutions. We refer to \cite{Ja-76, KuOr-11, Or-96, Or-01, Or-11, QiTo-04, QiTo-05}, where different applications of the successor map can be found.

In this paper we are mainly interested in Hamiltonian systems. Assuming the Hamiltonian function to be of class $\mathcal{C}^1$ in all the three variables, we show how to deal with a possible multivalued successor map. This situation arises when the uniqueness property of the solutions of initial value problems associated with system \eqref{syst-fg} is not assumed. 

In Section~\ref{section-2} we show that, thanks to a suitable symplectic change of variables, the successor map can be interpreted as the Poincar\'e map of an equivalent Hamiltonian system.
In this way, in Section~\ref{section-3}, we are able to apply a variant of the Poincar\'e--Birkhoff theorem \cite{FoUr-17} where the above mentioned uniqueness property is not required.
Whether the regularity hypothesis on the Hamiltonian function could be weakened by assuming only the $\mathcal{C}^1$-smoothness in the $(x,y)$ variables remains an open question; in such a case, a different approach would probably be needed. Section~\ref{section-4} is devoted to presenting a set of minimal hypotheses ensuring that the successor map associated with a general planar system is well defined.

In Section~\ref{section-5} we provide an application of our version of the Poincar\'e--Birkhoff theorem to the periodic problem associated with the scalar second order equation
\begin{equation}\label{ela}
\ddot x + \lambda g(t,x) = 0,
\end{equation}
where $g$ is continuous, $T$-periodic in $t$, and satisfies the superlinear growth condition
\begin{equation*}
\lim_{x\to \pm \infty} \frac{g(t,x)}{x} = +\infty, 
\quad \text{uniformly in $t$.}
\end{equation*}
Under these assumptions, we prove the existence of an arbitrarily large number of periodic solutions provided that $\lambda>0$ is sufficiently small. More precisely, as a consequence of Theorem~\ref{THM_lambda}, we have the following.

\begin{corollary}\label{cor-intro}
For any integer $N$, there exists $\lambda_N>0$ such that, for every $\lambda \in\mathopen{]}0,\lambda_N\mathclose{]}$, equation \eqref{ela} has at least $N$ periodic solutions.
\end{corollary}

Let us recall that the existence of an infinite number of periodic solutions for the equation $\ddot x +g(t,x)=0$, with the above superlinear growth, was first provided by Jacobowitz in \cite{Ja-76} and Hartman in \cite{Ha-77}, assuming the uniqueness of the solutions to initial value problems and that $x=0$ is an equilibrium (see also \cite{FoSf-16} for a higher dimensional generalization). In our Theorem~\ref{THM_lambda} we do not need these assumptions. We also refer, e.g., to \cite{BaBe-84,Bo-12,DiZa-92,FoMaZa-93,FuLo-75,Gi-preprint,Gi-23,Mo-65} for some related results.

\section{The successor map for Hamiltonian systems}\label{section-2}

Let us consider the planar Hamiltonian system
\begin{equation*}
J\dot z=\nabla_z H(t,z), 
\quad \text{where $J=\begin{pmatrix}
0 & -1\\
1 & 0
\end{pmatrix}$,}
\end{equation*}
ruled by a Hamiltonian function $H\colon \mathbb{R}\times\mathbb{R}^2\to \mathbb{R}$ of class $\mathcal{C}^1$, or equivalently, writing $z=(x,y)$,
\begin{equation}\label{HamSyst}
\begin{cases}
\phantom{-}\dot x = \partial_y H(t,x,y),
\\
-\dot y = \partial_x H(t,x,y).
\end{cases}
\end{equation}
In addition, we assume the following two conditions
\begin{enumerate}[leftmargin=30pt,labelsep=8pt,label=\textup{$(H_{1})$}]
\item $\langle \nabla_z H(t,z),z \rangle >0$, for every $(t,z)\in\mathbb{R}\times(\mathbb{R}^{2}\setminus\{0\})$,
\label{posit}
\end{enumerate}
\begin{enumerate}[leftmargin=30pt,labelsep=8pt,label=\textup{$(H_{2})$}]
\item $\displaystyle \lim_{|z|\to +\infty} H(t,z)=+\infty$, uniformly in $t\in\mathbb{R}$.
\label{susu}
\end{enumerate}
From \ref{posit} we deduce that
\begin{equation*}
\nabla_z H(t,0)=0, \quad \text{for every $t\in\mathbb{R}$,}
\end{equation*}
i.e., the origin is an equilibrium.

\begin{remark}\label{rem-2.4}
The hypothesis \ref{posit} may appear too restrictive for the applications, where usually the inequality holds only for large values of $|z|$.
However, this difficulty can be often overcome in practice.
For instance, when looking for $T$-periodic solutions in such a situation, one could try to modify the system in a neighborhood of the origin so to recover \ref{posit}. Then, once these $T$-periodic solutions are found, it will be necessary to prove that they lie in the region where the system has not been modified, so that they indeed solve the original system. This strategy will be used, for instance, in Section~\ref{section-5}.
\hfill$\lhd$
\end{remark}

\begin{remark}\label{increasing}
Assumption \ref{posit} implies that, for every $v\in\mathbb{R}^{2}\setminus\{0\}$ and every $t\in\mathbb{R}$, the function $\lambda \mapsto H(t,\lambda v)$ is strictly increasing in $\lambda\in\mathopen{[}0,+\infty\mathclose{[}$.
\hfill$\lhd$
\end{remark}

Without loss of generality, we can assume
\begin{equation}\label{zeroatzero}
H(t,0)=0, \quad \text{for every $t\in\mathbb{R}$.}
\end{equation}

\subsection{From times to angles}\label{section-2.1}

The change of variables \eqref{modpolcor} leads to the Hamiltonian system
\begin{equation*}
J\dot \zeta=\nabla_\zeta \HH(t,\zeta),
\end{equation*}
that is, writing $\zeta=(\theta,\rho)$,
\begin{equation}\label{HamSyst-H}
\begin{cases}
\phantom{-}\dot\theta = \partial_\rho \HH(t,\theta,\rho),
\\
-\dot\rho = \partial_\theta \HH(t,\theta,\rho).
\end{cases}
\end{equation}
It is ruled by the Hamiltonian $\HH \colon \UU\to \mathbb{R}$, with $\UU=\mathbb{R}\times\mathbb{R}\times\mathopen{]}0,+\infty\mathclose{[}$, defined as
\begin{equation}\label{HHH}
\HH(t,\theta,\rho)= H\bigl(t, \sqrt{2\rho} \sin \theta, \sqrt{2\rho} \cos \theta\bigr).
\end{equation}
Notice that $\HH$ is $2\pi$-periodic in the variable $\theta$.
A simple computation gives
\begin{align*}
\dot\theta&=\frac{\dot xy-x\dot y}{x^2+y^2} = \frac{x\partial_x H+y\partial_y H}{x^2+y^2} = \frac{\langle \nabla_z H , z \rangle}{|z|^2},
\\
\dot\rho &= x\dot x+y\dot y= x \partial_y H -y\partial_x H = \langle \nabla_z H , Jz \rangle.
\end{align*}

In order to perform another change of variables, we introduce the map $\Psi$.

\begin{proposition}\label{prop-map-Psi}
There is a function $\Psi=\Psi(\vartheta,\tau,h) \colon \UU\to \mathbb{R}$ of class $\mathcal{C}^1$, $2\pi$-periodic in the variable~$\vartheta$, such that
\begin{align}
\label{id1}
\Psi(\theta,t,\HH(t,\theta,\rho))=\rho,
&\quad \text{for every $(t,\theta,\rho)\in\UU$,}
\\
\label{id2}
\HH(\tau,\vartheta,\Psi(\vartheta,\tau,h))=h,
&\quad \text{for every $(\vartheta,\tau,h)\in\UU$.}
\end{align}
Moreover, if the Hamiltonian function $H$ is $T$-periodic in the variable $t$, then the function $\Psi$ is $T$-periodic in the variable $\tau$.
\end{proposition}

\begin{proof}
By Remark~\ref{increasing}, we get
\begin{equation*}
\partial_\rho \HH(t,\theta,\rho)>0,
\quad
\text{for every $(t,\theta,\rho)\in\UU$.}
\end{equation*}
Hence, the existence of $\Psi$ follows by an application of the implicit function theorem. 
Since \ref{posit}, \ref{susu} and \eqref{zeroatzero} hold, the function $\Psi(\theta,t,\cdot)\colon \mathopen{]}0,+\infty\mathclose{[} \to \mathopen{]}0,+\infty\mathclose{[}$ is a bijection for every $(\theta,t)\in\mathbb{R}^2$. Hence $\Psi$ is defined on $\UU$.
Then, for every $(\vartheta,\tau,h)\in\UU$ we can find $\rho>0$ such that $\HH(\tau,\vartheta,\rho)=h$, so that
\begin{align*}
\Psi(\vartheta,\tau,h)&=
\Psi(\vartheta,\tau,\HH(\tau,\vartheta,\rho)) = \rho 
=\Psi(\vartheta+2\pi,\tau,\HH(\tau,\vartheta+2\pi,\rho))
\\
&=\Psi(\vartheta+2\pi,\tau,\HH(\tau,\vartheta,\rho))
=\Psi(\vartheta+2\pi,\tau,h),
\end{align*}
and, if $H$ is $T$-periodic in $t$ (and so also $\HH$),
\begin{align*}
\Psi(\vartheta,\tau,h)
&=\Psi(\vartheta,\tau,\HH(\tau,\vartheta,\rho)) = \rho 
=\Psi(\vartheta,\tau+T,\HH(\tau+T,\vartheta,\rho))
\\
&=\Psi(\vartheta,\tau+T,\HH(\tau,\vartheta,\rho))
=\Psi(\vartheta,\tau+T,h).
\end{align*}
The previous computations
provide the periodicity properties in the statement.
\end{proof}

Let $\Psi$ be as in Proposition~\ref{prop-map-Psi}. We can introduce the Hamiltonian system
\begin{equation*}
Jw'=\nabla_w \Psi(\vartheta,w),
\end{equation*}
that is,  writing $w=(\tau,h)$,
\begin{equation}\label{HamSyst-Psi}
\begin{cases}
\phantom{-}\tau'
= \partial_h \Psi(\vartheta,\tau,h),
\\
-h'
= \partial_\tau \Psi(\vartheta,\tau,h).
\end{cases}
\end{equation}
Let us consider a solution $\zeta=(\theta,\rho) \colon \mathcal I \to \mathbb{R}\times \mathopen{]}0,+\infty\mathclose{[}$ of system \eqref{HamSyst-H}. It is well-known that
\begin{equation}\label{Hder}
\dfrac{\mathrm{d}}{\mathrm{d}t}\HH(t,\theta(t),\rho(t))=\partial_t \HH(t,\theta(t),\rho(t)).
\end{equation}
By \ref{posit}, the function $\theta \colon\II\to\JJ$ has an inverse function $\tau \colon\JJ\to \II$. We define $h \colon\JJ\to\mathbb{R}$ as
\begin{equation}\label{hdef}
h(\vartheta)=\HH(\tau(\vartheta),\vartheta,\rho(\tau(\vartheta)))= [\HH(\cdot,\theta(\cdot),\rho(\cdot))\circ \tau](\vartheta).
\end{equation}
We now prove that, if $(\theta,\rho)$ is a solution of \eqref{HamSyst-H}, then the couple $(\tau,h)$ solves \eqref{HamSyst-Psi}.
Indeed, deriving \eqref{id2} with respect to the variable $h$ we get
\begin{equation*}
\partial_\rho \HH(\tau,\vartheta,\Psi(\vartheta,\tau,h)) \cdot \partial_h\Psi(\vartheta,\tau,h) =1.
\end{equation*}
Hence, recalling \eqref{HamSyst-H},
\begin{equation}\label{AA2}
\tau'(\vartheta)
= \frac{1}{\dot \theta(\tau(\vartheta))} =  \frac{1}{\partial_\rho \HH(\tau(\vartheta),\vartheta,\Psi(\vartheta,\tau(\vartheta),h(\vartheta)))} 
= \partial_h\Psi(\vartheta,\tau(\vartheta),h(\vartheta)),
\end{equation}
so that the first equation in \eqref{HamSyst-Psi} is satisfied.
Moreover, differentiating in \eqref{id1} with respect to the variable $t$, we get
\begin{equation}\label{BB1}
    \partial_\tau \Psi(\theta,t,\HH(t,\theta,\rho))
    + \partial_h \Psi(\theta,t,\HH(t,\theta,\rho)) \cdot
    \partial_t \HH(t,\theta,\rho) = 0.
\end{equation}
Then, recalling in the order \eqref{hdef}, \eqref{AA2}, \eqref{Hder}, and \eqref{BB1}, we deduce
\begin{align*}
    h'(\vartheta)
    &
    = 
    \left[ \frac{1}{\dot \theta(t)} \cdot \frac{\mathrm{d}}{\mathrm{d}t} \HH(t,\theta(t),\rho(t))\right]_{t=\tau(\vartheta)}\\
    &=\partial_h \Psi(\vartheta,\tau(\vartheta),h(\vartheta))\cdot
    \partial_t    \HH(\tau(\vartheta),\vartheta,\rho(\tau(\vartheta)))
    \\
    &=\partial_h \Psi(\vartheta,\tau(\vartheta),\HH(\tau(\vartheta),\vartheta,\rho(\tau(\vartheta))))\cdot
    \partial_t    \HH(\tau(\vartheta),\vartheta,\rho(\tau(\vartheta)))\\
    &=-\partial_\tau \Psi(\vartheta,\tau(\vartheta),\HH(\tau(\vartheta),\vartheta,\rho(\tau(\vartheta)))) = -\partial_\tau \Psi(\vartheta,\tau(\vartheta),h(\vartheta)),
\end{align*}
thus proving the validity of the second equation in \eqref{HamSyst-Psi}.

Analogously one can prove that, if $w=(\tau,h)$ is a solution of system \eqref{HamSyst-Psi}, with $h>0$, 
then $\zeta=(\theta,\rho)$  is a solution of \eqref{HamSyst-H}, taking $\theta$ as the inverse of $\tau$ and $\rho(t)=
\Psi(\theta(t),t,h(\theta(t)))$.

\subsection{Successor map vs. Poincar\'e map}\label{section-2.4}

We now assume that the successor map $\SuccMap$ for the Hamiltonian system \eqref{HamSyst} is well-defined and show that it corresponds to the Poincar\'e map for system \eqref{HamSyst-Psi}.
To this aim, in order to avoid cumbersome notation, we provisionally assume
that for every $(t_0,x_0,y_0)\in\mathbb{R}\times\mathbb{R}^{2}$ the Cauchy problem
\begin{equation*}
\begin{cases}
\, \dot x = \partial_y H(t,x,y),
&
-\dot y = \partial_x H(t,x,y),
\\
\, x(t_0)=x_0,
&
y(t_0)=y_0,
\end{cases}
\end{equation*}
has a unique solution, denoted by 
\begin{equation*}
z(t;t_0,x_0,y_0)=\bigl( x(t;t_0,x_0,y_0),y(t;t_0,x_0,y_0)\bigr).
\end{equation*}

Let us also consider the Cauchy problem associated with system \eqref{HamSyst-Psi}, i.e.,
\begin{equation*}
\begin{cases}
\, \tau'
    = \partial_h \Psi(\vartheta,\tau,h), &
    -h'
    = \partial_\tau \Psi(\vartheta,\tau,h),  
    \\
\, \tau(\vartheta_0)=\tau_0, &
    \phantom{-} h(\vartheta_0)=h_0,
\end{cases}
\end{equation*}
and denote its solution by
\begin{equation*}
w(\vartheta;\vartheta_0, \tau_0,h_0)
=
\bigl(
h(\vartheta;\vartheta_0, \tau_0,h_0),
\tau(\vartheta;\vartheta_0, \tau_0,h_0)
\bigr).
\end{equation*}
The Poincar\'e map $\PP_{2\pi} \colon \Omega \to \mathbb{R}^2$ is defined as
\begin{equation*}
\PP_{2\pi}(\tau_0,h_0) = w(2\pi; 0,\tau_0,h_0),
\end{equation*}
where $\Omega$ is the subset of $\mathbb{R}^2$ containing those points $(\tau_0,h_0)$ such that the solution $w(\,\cdot\,;0, \tau_0,h_0)$ is defined at least in the interval $[0,2\pi]$.

By the above arguments, we can verify that the trajectory $w(\vartheta;0, \tau_0,h_0)$ of system \eqref{HamSyst-Psi} corresponds to the trajectory $z(t;t_0,0,y_0)$ of system \eqref{HamSyst} with
\begin{equation*}
t_0=\tau_0
\quad\mbox{and}\quad
y_0=\Psi(0,\tau_0,h_0).
\end{equation*}
In particular, since we are assuming that the successor map $\SuccMap$ is well-defined, we have $\Omega=\mathbb{R}^2$ and, recalling the notation in \eqref{SY},
\begin{equation*}
\PP_{2\pi}(t_0,h_0)= \bigl(\TT\big(t_0,\Psi(0,t_0,h_0)\bigr),
\HH\big(t_1,2\pi,\YY(t_0,\Psi(0,t_0,h_0))\bigr)\bigr).
\end{equation*}

Summing up, we have the following.

\begin{proposition}\label{prop-2.3}
The equivalence
\begin{equation*}
(t_1,y_1)=\SuccMap(t_0,y_0) \iff 
(t_1,h_1)=\PP_{2\pi}(t_0,h_0)
\end{equation*}
holds, with $h_0=\HH(t_0,0,y_0)$ and $h_1=\HH(t_1,0,y_1)$.
\end{proposition}

\begin{remark}\label{rem-2.5}
The hypothesis \ref{susu} is rather natural in the applications we have in mind. Notice that it has been used only to simplify the choice of the domain of the function $\Psi$.
\hfill$\lhd$
\end{remark}

To conclude this section, we emphasize  that all the above discussion may be rephrased also in the case when there is no uniqueness for initial value problems. Clearly enough, the notation will be accordingly interpreted.

\section{A Poincaré--Birkhoff theorem for multivalued successor maps}\label{section-3}

We now present a version of the Poincaré--Birkhoff theorem in the context of the successor map $\mathcal{S}$. Notice that we do not require uniqueness for the solutions of initial value problems.

\begin{theorem}\label{PB-thm}
Let $H\colon \mathbb{R}\times\mathbb{R}^{2} \to \mathbb{R}$ be a function of class $\mathcal{C}^{1}$, $T$-periodic in $t$, satisfying \ref{posit} and \ref{susu}.
Assume that the successor map $\SuccMap=(\TT,\YY)$ is well-defined on $\R\times\mathopen{]}0,+\infty\mathclose{[}$ and that there are two positive constants $\alpha<\beta$ and two positive integers $m$ and $k$ such that
\begin{equation}
\label{twist-xy}
\TT^m(t_0,\alpha)-t_0> kT,
\qquad
\TT^m(t_0,\beta)-t_0 < kT,
\end{equation}
for every $t_0\in[0,T]$.
Then, there exist two distinct $kT$-periodic solutions $z^{(i)}=(x^{(i)},y^{(i)})$ of \eqref{HamSyst}, with $i=1,2$, such that
\begin{equation}\label{Rot-x}
\mathrm{Rot}(z^{(i)}(t),\mathopen{[}0,kT \mathclose{]}) = m,
\end{equation}
and
\begin{equation}\label{loc-x}
\bigl{\{} y^{(i)}(t) \colon t\in \mathopen{[}0,T \mathclose{[}\bigr{\}}\cap \mathopen{]}\alpha,\beta\mathclose{[}\neq \emptyset.
\end{equation}
\end{theorem}

\begin{remark}\label{rem-multivalued}
In the statement of Theorem~\ref{PB-thm} we allow the successor map $\SuccMap=(\TT,\YY)$ to be multivalued, and the iterates $\SuccMap^m=(\TT^m,\YY^m)$ as well. In this case, condition \eqref{twist-xy} says that, for every $(\xi,\upsilon)\in \SuccMap^m(t_0,\alpha)$, one has $\xi-t_0> kT$, and for every $(\xi,\upsilon)\in \SuccMap^m(t_0,\beta)$, one has $\xi-t_0< kT$. In the following, we will implicitly agree with such an interpretation, in order to simplify the notation.
\hfill$\lhd$
\end{remark}

\begin{remark}\label{sol-dist}
The two solutions we find in Theorem~\ref{PB-thm} are indeed \textit{distinct} since we will show that $z^{(1)}(\cdot) \neq z^{(2)}(\cdot + j T)$ for every $j\in\mathbb{Z}$.
\hfill$\lhd$
\end{remark}

\begin{proof}[Proof of Theorem~\ref{PB-thm}]
We define
\begin{equation}\label{def-ab}
a(t)=H(t,0,\alpha)
\quad \text{ and } \quad
b(t)=H(t,0,\beta).
\end{equation}
From Remark~\ref{increasing}
we get $a(t)<b(t)$ for every $t\in\mathbb{R}$.

Performing the change of variables described in Section~\ref{section-2.1} which leads to the Hamiltonian system \eqref{HamSyst-Psi}, we see that every solution $(\tau,h)$ of \eqref{HamSyst-Psi} with initial condition
$h(0)\in[a(\tau(0)),b(\tau(0))]$
is defined in the interval $[0,2\pi m]$ and satisfies
\begin{equation*}
\begin{cases}
\, \tau(2\pi m)-\tau(0) > kT,
&\text{if $h(0)=a(\tau(0))$,}
\\
\, \tau(2\pi m)-\tau(0) <  kT,
&\text{if $h(0)=b(\tau(0))$.}
\end{cases}
\end{equation*}
By the periodicity properties of  $\Psi$ stated in Proposition~\ref{prop-map-Psi}, 
we can apply the Poincar\'e--Birkhoff theorem in the version stated in~\cite[Theorem~6.2]{FoUr-17}
to the auxiliary system with Hamiltonian function
\begin{equation*}
\overline\Psi(\vartheta,\tau,h)= \Psi\left(\vartheta,\tau+\frac{kT}{2m\pi}\vartheta,h\right)-\frac {kT}{2m\pi} h,
\end{equation*}
thus
obtaining two solutions $w^{(i)}=(\tau^{(i)},h^{(i)})$ of \eqref{HamSyst-Psi} 
such that $\tau^{(i)}(0)\in\mathopen{[}0,T\mathclose{[}$, $h^{(i)}(0)\in\mathopen{]}a(\tau^{(i)}(0)),b(\tau^{(i)}(0))\mathclose{[}$ and
\begin{equation*}
\tau^{(i)}(\vartheta+2\pi m)=\tau^{(i)}(\vartheta)+ kT,
\quad
h^{(i)}(\vartheta+2\pi m)=h^{(i)}(\vartheta),
\quad\text{for every $\vartheta\in\R$.}
\end{equation*}
Moreover, 
\begin{equation}\label{quelasora}
\tau^{(1)}(\cdot+ 2\pi \ell_1)+ j_1 T \neq \tau^{(2)}(\cdot+2\pi \ell_2) +j_2 T, \quad \text{for every $j_1,j_2,\ell_1,\ell_2\in\mathbb{Z}$.}
\end{equation}
This follows from the proof of the Poincar\'e--Birkhoff theorem in \cite{FoUr-17}, as explained in \cite[p.~2352]{FoTo-23}.

Using Proposition~\ref{prop-map-Psi} and the change of variables \eqref{modpolcor}, we can recover two solutions $z^{(i)}=(x^{(i)},y^{(i)})$ of system \eqref{HamSyst} which are $kT$-periodic and
satisfy \eqref{Rot-x}. Writing $z^{(i)}$ in the generalized polar coordinates  $(\theta^{(i)},\rho^{(i)})$, as in \eqref{modpolcor}, we have that  $\theta^{(i)}(t)$  is the inverse of $\tau^{(i)}(\vartheta)$. So, passing to the inverse functions, \eqref{quelasora} reads as
\begin{equation*}
\theta^{(1)}(\cdot -j_1 T)- 2\pi \ell_1 \neq \theta^{(2)}(\cdot -j_2 T)- 2\pi \ell_2, \quad \text{for every $j_1,j_2,\ell_1,\ell_2\in\mathbb{Z}$.}
\end{equation*}
We have thus proved that
$z^{(1)}(\cdot)\neq z^{(2)}(\cdot + jT)$, for every $j\in\mathbb Z$.

Recalling \eqref{HHH} and \eqref{def-ab}, we get
\begin{equation*}
\HH(\tau^{(i)}(0),0,\tfrac12 \alpha^2) < h^{(i)}(0) <\HH(\tau^{(i)}(0),0,\tfrac12 \beta^2),
\end{equation*}
and so, by Proposition~\ref{prop-map-Psi}, we have
\begin{equation*}
\tfrac12 \alpha^2 < \Psi(0,\tau^{(i)}(0),h^{(i)}(0)) < \tfrac12 \beta^2.
\end{equation*}
Since 
\begin{equation*}
\tfrac12 [y^{(i)}(\tau^{(i)}(0))]^2 = \rho(\tau^{(i)}(0))=\Psi(0,\tau^{(i)}(0),h^{(i)}(0)),
\end{equation*}
we get \eqref{loc-x}. The proof is thus completed.
\end{proof}

\section{Sufficient conditions for the existence of the successor map}\label{section-4}

With the aim of well-defining the successor map $\mathcal{S}$ associated with a general planar system 
\begin{equation}\label{syst-fg-bis}
\begin{cases}
\, \phantom{-}\dot x=f(t,x,y),
\\
\, -\dot y=g(t,x,y),
\end{cases}
\end{equation}
where $f,g \colon\mathbb{R}\times\mathbb{R}^2\to\mathbb{R}$ are continuous functions,
we introduce the following hypotheses.

\begin{enumerate}[leftmargin=30pt,labelsep=8pt,label=\textup{$(A_{1})$}]
\item $g(t,x,y)x+f(t,x,y)y>0$, for every $(t,x,y)\in\mathbb{R}\times(\mathbb{R}^2\setminus\{0\})$.
\label{hp-A-1}
\end{enumerate}
Notice that \ref{hp-A-1} corresponds to \ref{posit} in the case of Hamiltonian systems. In particular, the origin is an equilibrium. 
Recalling the modified polar coordinates introduced in \eqref{modpolcor},
assumption \ref{hp-A-1} implies that, for any solution of system \eqref{syst-fg-bis},
\begin{equation*}
\dot\theta(t)=\frac{g(t,x(t),y(t))x(t)+f(t,x(t),y(t))y(t)}{x^2(t)+y^2(t)}>0,
\end{equation*}
as long as the solution exists and remains away from the origin.

\begin{enumerate}[leftmargin=30pt,labelsep=8pt,label=\textup{$(A_{2})$}]
\item There exist $R>0$ and $c_1>0$ such that
\begin{equation}\label{gxfy}
g(t,x,y)x+f(t,x,y)y\geq c_1(x^2+y^2),
\end{equation}
for every $(t,x,y)\in\mathbb{R}\times\mathbb{R}^2$ with $x^2+y^2\ge R^2$.
\label{hp-A-2}
\end{enumerate}
We observe that hypothesis \ref{hp-A-2} implies that the solutions of \eqref{syst-fg-bis} rotate decisively when they are far from the origin.

\begin{enumerate}[leftmargin=30pt,labelsep=8pt,label=\textup{$(A_{3})$}]
\item There exist two continuous functions $\phi,\psi \colon \mathopen{[}0,+\infty\mathclose{[} \to \mathopen{[}0,+\infty\mathclose{[}$ such that
\begin{equation*}
|f(t,x,y)|\le\phi(|y|),
\quad
|g(t,x,y)|\le\psi(|x|),
\end{equation*}
for every $(t,x,y)\in\mathbb{R}\times\mathbb{R}^2$.
\label{hp-A-3}
\end{enumerate}
Clearly, hypothesis \ref{hp-A-3} is satisfied if $f(t,x,y)$ does not depend on $x$ and $g(t,x,y)$ does not depend on $y$.
The next assumption describes the behaviour near the axes.

\begin{enumerate}[leftmargin=30pt,labelsep=8pt,label=\textup{$(A_{4})$}]
\item \label{hp-A-4} There exist $\delta>0$ such that, for every $(t,x,y)\in\mathbb{R}\times\mathbb{R}^{2}$,
\begin{align}
\label{A4-a}
&\bigl{[} \; |x|\leq\delta,\;xy\geq0,\;y\neq0 \;\bigr{]}
\quad\Rightarrow\quad f(t,x,y)y>0,
\\
\label{A4-b}
&\bigl{[}\;|y|\leq\delta,\;xy\leq0,\;x\neq0\;\bigr{]}
\quad\Rightarrow\quad g(t,x,y)x>0.
\end{align}
\end{enumerate}
Finally, we need to control the vector field in some regions which are far from the origin.

\begin{enumerate}[leftmargin=30pt,labelsep=8pt,label=\textup{$(A_{5})$}]
\item \label{hp-A-5} There exist $D>0$ such that, for every $(t,x,y)\in\mathbb{R}\times\mathbb{R}^{2}$,
\begin{align}
\label{A5-a}
&\bigl{[}\; |x|\geq D,\;|y|\geq D,\;xy\leq0\;\bigr{]}\quad\Rightarrow\quad f(t,x,y)y>0,
\\
\label{A5-b}
&\bigl{[}\;|x|\geq D,\;|y|\geq D,\;xy\geq0\;\bigr{]}\quad\Rightarrow\quad g(t,x,y)x>0.
\end{align}
\end{enumerate}
See Figure~\ref{fig-01} for a graphical representation of the hypotheses
\ref{hp-A-4} and \ref{hp-A-5}.

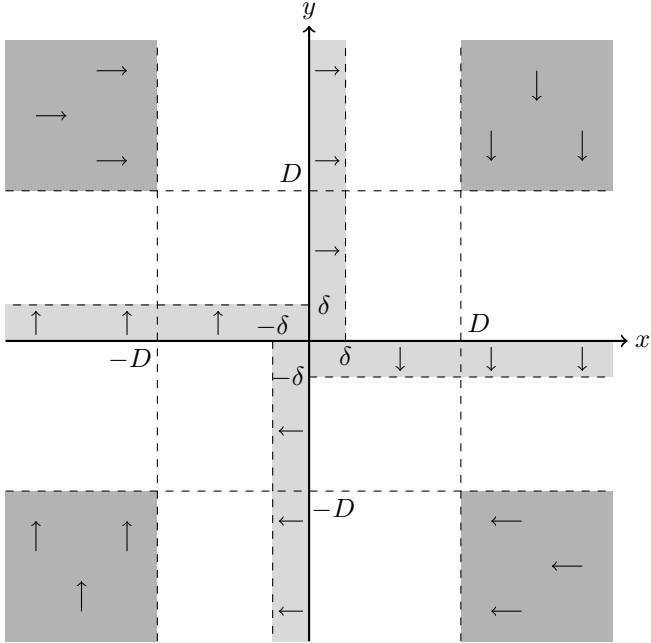
\begin{figure}[t]
\centering
\begin{tikzpicture}[x=4mm,y=4mm]
\fill[fill=FFS-1] (0,0) -- (1.2,0) -- (1.2,10) -- (0,10) -- (0,0);
\fill[fill=FFS-1] (0,0) -- (0,-1.2) -- (10,-1.2) -- (10,0) -- (0,0);
\fill[fill=FFS-1] (0,0) -- (-1.2,0) -- (-1.2,-10) -- (0,-10) -- (0,0);
\fill[fill=FFS-1] (0,0) -- (0,1.2) -- (-10,1.2) -- (-10,0) -- (0,0);
\fill[fill=FFS-2] (5,5) -- (10,5) -- (10,10) -- (5,10) -- (5,5);
\fill[fill=FFS-2] (5,-5) -- (10,-5) -- (10,-10) -- (5,-10) -- (5,-5);
\fill[fill=FFS-2] (-5,-5) -- (-10,-5) -- (-10,-10) -- (-5,-10) -- (-5,-5);
\fill[fill=FFS-2] (-5,5) -- (-10,5) -- (-10,10) -- (-5,10) -- (-5,5);
\draw[->,line width=0.8pt] (-10,0) -- (10.5,0);
\draw[->,line width=0.8pt] (0,-10) -- (0,10.5);
\draw (11,0) node {$x$};
\draw (0,11) node {$y$};
\draw[dashed] (1.2,0) -- (1.2,10);
\draw[dashed] (0,-1.2) -- (10,-1.2);
\draw[dashed] (-1.2,0) -- (-1.2,-10);
\draw[dashed] (0,1.2) -- (-10,1.2);
\draw[dashed] (5,-10) -- (5,10);
\draw[dashed] (-5,-10) -- (-5,10);
\draw[dashed] (-10,5) -- (10,5);
\draw[dashed] (-10,-5) -- (10,-5);
\draw[->] (6,7) -- (6,6);
\draw[->] (9,7) -- (9,6);
\draw[->] (7.5,9) -- (7.5,8);
\draw[->] (-6,-7) -- (-6,-6);
\draw[->] (-9,-7) -- (-9,-6);
\draw[->] (-7.5,-9) -- (-7.5,-8);
\draw[->] (7,-6) -- (6,-6);
\draw[->] (7,-9) -- (6,-9);
\draw[->] (9,-7.5) -- (8,-7.5);
\draw[->] (-7,6) -- (-6,6);
\draw[->] (-7,9) -- (-6,9);
\draw[->] (-9,7.5) -- (-8,7.5);
\draw[->] (0.2,3) -- (1,3);
\draw[->] (0.2,6) -- (1,6);
\draw[->] (0.2,9) -- (1,9);
\draw[->] (-0.2,-3) -- (-1,-3);
\draw[->] (-0.2,-6) -- (-1,-6);
\draw[->] (-0.2,-9) -- (-1,-9);
\draw[->] (3,-0.2) -- (3,-1);
\draw[->] (6,-0.2) -- (6,-1);
\draw[->] (9,-0.2) -- (9,-1);
\draw[->] (-3,0.2) -- (-3,1);
\draw[->] (-6,0.2) -- (-6,1);
\draw[->] (-9,0.2) -- (-9,1);
\draw (0.5,1.2) node {$\delta$};
\draw (1.2,-0.5) node {$\delta$};
\draw (-0.7,-1.2) node {$-\delta$};
\draw (-1.2,0.5) node {$-\delta$};
\draw (-0.6,5.6) node {$D$};
\draw (5.6,0.6) node {$D$};
\draw (0.8,-5.6) node {$-D$};
\draw (-5.9,-0.6) node {$-D$};
\end{tikzpicture}
\caption{Qualitative representation of hypotheses \ref{hp-A-4} (concerning the area colored in light gray)
and \ref{hp-A-5} (concerning the area colored in dark gray).}
\label{fig-01}
\end{figure}

Under the above assumptions, we will show that any (non-zero) solution of \eqref{syst-fg-bis} performs a complete rotation around the origin in finite time. More precisely, we have the following (cf.~Figure~\ref{fig-02}).

\begin{theorem}\label{successor_thm}
Let \ref{hp-A-1}--\ref{hp-A-5} hold true. Then, every solution $(x,y)$ of \eqref{syst-fg-bis} with initial conditions $x(t_0)=0$, $y(t_0)=y_0>0$ is defined on an interval $[t_0,t_1]$, and there exist $t_0'$, $t_0''$, $t_0'''$ with $t_0<t_0'<t_0''<t_0'''<t_1$ such that 
\begin{align*}
&x(t)>0, \; y(t)>0, \quad\text{for every $t\in\mathopen{]}t_0,t_0'\mathclose{[}$,}
\\
&x(t_0')>0, \; y(t_0')=0,
\\
&x(t)>0, \; y(t)<0, \quad\text{for every $t\in\mathopen{]}t_0',t_0''\mathclose{[}$,}
\\
&x(t_0'')=0, \; y(t_0'')<0,
\\
&x(t)<0, \; y(t)<0, \quad\text{for every $t\in\mathopen{]}t_0'',t_0'''\mathclose{[}$,}
\\
&x(t_0''')<0, \; y(t_0''')=0,
\\
&x(t)<0, \; y(t)>0 \quad\text{for every $t\in\mathopen{]}t_0''',t_1\mathclose{[}$,}
\\
&x(t_1)=0, \; y(t_1)>0.
\end{align*}
\end{theorem}

\begin{figure}[t]
\centering
\begin{tikzpicture}[x=3mm,y=3mm]
\draw[->,line width=0.8pt] (-10,0) -- (10.5,0);
\draw (11.2,0) node {$x$};
\draw[->,line width=0.8pt] (0,-10) -- (0,10.5);
\draw (0,11.2) node {$y$};
\draw [color=gray,line width=0.9pt] plot [smooth] coordinates {(0,7.5) (3,8) (8,8) (5,3) (6,0) (7,-2) (6,-6) (0,-7) (-5,-6) (-4.5,-4) (-2,-2.5) (-3,0) (-4,1.5) (-4,3) (-5.5,6) (-2,6) (0,3.5)};
\draw[->,color=gray,line width=0.8pt] (3.8,8.1) -- (3.81,8.101);
\draw[->,color=gray,line width=0.8pt] (5.5,4) -- (5.4,3.85);
\draw[->,color=gray,line width=0.8pt] (6.95,-3) -- (6.93,-3.1);
\draw[->,color=gray,line width=0.8pt] (2.5,-6.82) -- (2.35,-6.84);
\draw[->,color=gray,line width=0.8pt] (-4.7,-4.3) -- (-4.685,-4.28);
\draw[->,color=gray,line width=0.8pt] (-4.02,2) -- (-4.01,2.05);
\draw[->,color=gray,line width=0.8pt] (-2.4,6.1) -- (-2.3,6.08);
\draw (-2,8) node {$t=t_{0}$};
\draw (4.6,-1) node {$t=t_{0}'$};
\draw (-2,-8) node {$t=t_{0}''$};
\draw (-4.9,-1) node {$t=t_{0}'''$};
\draw (1.8,4) node {$t=t_{1}$};
\fill (0,7.5) circle (1.2pt);
\fill (6,0) circle (1.2pt);
\fill (0,-7) circle (1.2pt);
\fill (-3,0) circle (1.2pt);
\fill (0,3.5) circle (1.2pt);
\end{tikzpicture}
\caption{Qualitative representation of the statement of Theorem~\ref{successor_thm}.}
\label{fig-02}
\end{figure}
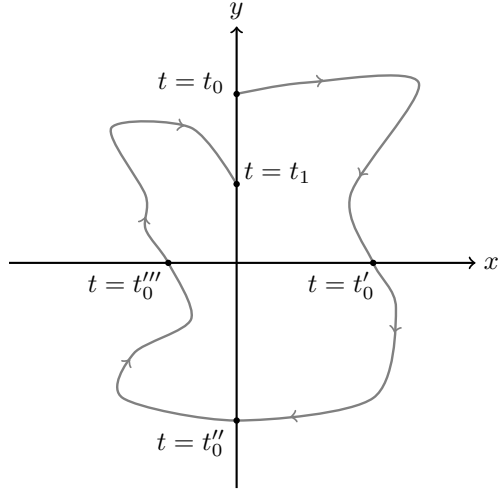

\begin{proof}
We focus our attention on the first quadrant $\mathopen{[}0,+\infty\mathclose{[}\times\mathopen{[}0,+\infty\mathclose{[}$, as the situation is symmetric in the other three quadrants.

Let us consider a solution $(x,y)$ of \eqref{syst-fg-bis} starting with $x(t_0)=0$, $y(t_0)=y_0>0$, and analyse its behaviour for $t>t_0$. It immediately moves away from the $y$-axis by going to the right, by \ref{hp-A-1} (or \eqref{A4-a}). In a small positive time, $(x,y)$ moves to a certain distance $\delta'\in\mathopen{]}0,\delta\mathclose{]}$ (depending on the solution) from the $y$-axis and cannot get closer anymore, by \eqref{A4-a}. In this time, it could already reach the $x$-axis; in this case we are done.

If $(x,y)$ continues its journey in the first quadrant without reaching the $x$-axis, it could explode in finite time. However, this cannot happen while $x$ remains bounded, since $\dot y$ would be bounded as well, by \ref{hp-A-3}. Moreover, it cannot happen while $y$ remains bounded, since $\dot x$ would be bounded, again by \ref{hp-A-3}. Hence, it has to be that both $x(t)\to+\infty$ and $y(t)\to+\infty$ as $t\to+\infty$. Then, there exists a $\bar t>t_0$ such that $x(t)\geq D$ and $y(t)\geq D$, for every $t\geq \bar t$. By \eqref{A5-b}, $y$ is strictly decreasing, hence $y(t)\leq y(\bar t)$ for every $t\geq\bar t$, so $y(t)$ remains bounded, which is impossible, by the above considerations.

At last, we show that the solution cannot remain in the first quadrant for all $t\in \mathopen{]}t_0,+\infty\mathclose{[}$.
By \ref{hp-A-1}, using compactness and continuity, there exists a constant $c'_1\in\mathopen{]}0,c_1\mathclose{]}$ such that 
\begin{align*}
g(t,x,y)x+f(t,x,y)y \geq c'_1 (x^2+y^2),
\quad 
&\text{for every $(t,x,y)\in\mathbb{R}\times\mathbb{R}^{2}$}
\\
&\text{with $xy\geq0$, $x\geq\delta'$ and $x^2+y^2\leq R^2$,}
\end{align*}
where $c_1$ and $R$ are the constants appearing in \ref{hp-A-2}. Hence, using also \ref{hp-A-2},
\begin{align*}
g(t,x,y)x+ f(t,x,y)y\ge c'_1(x^2+y^2),
\quad 
&\text{for every $(t,x,y)\in\mathbb{R}\times\mathbb{R}^{2}$}
\\
&\text{with $xy\ge0$, $x\ge\delta'$.}
\end{align*}
Recalling that $x(t)\geq \delta'$ from some $t$ onwards, this means that the solution must turn decisively and therefore reach the $x$-axis in finite time.
The proof is complete.
\end{proof}

\begin{remark}
We notice that condition \ref{hp-A-2} can be rewritten using the modified polar coordinates \eqref{modpolcor} as $\dot\theta(t) \geq c_1>0$. Moreover, we stress that this condition can be improved by assuming the existence of a continuous function $\eta \colon \mathopen{[}0,2\pi\mathclose{]}\to\mathopen{]}0,+\infty\mathclose{[}$ such that $\dot\theta(t) \geq \eta(\theta(t))$ and
\begin{equation*}
\int_0^{2\pi} \frac{\mathrm{d}s}{\eta(s)} <+\infty.
\end{equation*}
Indeed, assuming that the solution remains in the first quadrant for every $t>t_0$, one has that
\begin{equation*}
t - t_{0} \leq \int_{t_{0}}^{t} \dfrac{\dot\theta(t)}{\eta(\theta(t))}\,\mathrm{d}t = \int_{\theta(t_{0})}^{\theta(t)} \dfrac{\mathrm{d}s}{\eta(s)} \leq\int_{0}^{2\pi} \dfrac{\mathrm{d}s}{\eta(s)},
\end{equation*}
leading to a contradiction. 
\hfill$\lhd$
\end{remark}

\begin{remark}\label{Jac-rem}
In the particular case of the scalar second order equation 
\begin{equation*}
\ddot x+g(t,x)=0,
\end{equation*}
we can write it as system \eqref{syst-fg-bis} with $f(t,x,y)=y$ and $g(t,x,y)=g(t,x)$. 
In this case condition \ref{hp-A-3} is clearly satisfied. Moreover, if we assume
\begin{equation*}
g(t,x)x>0, \quad \text{for every $x\in\mathbb{R}\setminus\{0\}$,}
\end{equation*}
then also \ref{hp-A-1},
\ref{hp-A-4} and
\ref{hp-A-5} hold. Furthermore, if
\begin{equation*}
\liminf_{x\to\pm\infty} \frac{g(t,x)}{x} \geq c >0, \quad \text{uniformly in $t$,}
\end{equation*}
then also condition \ref{hp-A-2} is fulfilled.
Hence, by Theorem~\ref{successor_thm}, the successor map is well-defined. Notice that a possible superlinear growth of $g$ could compromise the possibility of defining the Poincar\'e map, as suggested by the results in \cite{CoUl-67}.
\hfill$\lhd$
\end{remark}

\section{A superlinear differential equation}\label{section-5}

In this section we focus our attention on the scalar second order differential equation
\begin{equation}\label{lambda-eq}
\ddot x + \lambda g(t,x) = 0,
\end{equation}
where $\lambda>0$ is a parameter, the function $g\colon \R\times \R\to\mathbb{R}$ is continuous and $T$-periodic in the variable~$t$.
We assume that
\begin{equation}\label{g-superlinear}
\lim_{x\to \pm \infty} \frac{g(t,x)}{x} = +\infty, 
\quad \text{uniformly in $t$.}
\end{equation}

Let us state our multiplicity result.

\begin{theorem}\label{THM_lambda}
For any pair $(m,k)$ of positive integers there exists $\lambda_{m,k}>0$ such that, for every $\lambda \in\mathopen{]}0,\lambda_{m,k}\mathclose{]}$, equation \eqref{lambda-eq} has a $kT$-periodic solution having exactly $2m$ simple zeros in the interval $\mathopen{[}0,kT\mathclose{[}$.
Moreover, if $\partial_t g\colon  \R\times\R \to \R$ exists and is continuous, there are at least two such solutions. 
\end{theorem}

It can be easily seen that Corollary~\ref{cor-intro} is an immediate consequence of the previous statement. Notice that the nodal properties of the periodic solutions can give information about their minimal period. For instance, taking $m=1$, the minimal period will surely be equal to $kT$.

\proof[Proof of Theorem~\ref{THM_lambda}]
We will first provide the proof when 
$\partial_t g$ exists and is continuous.
Notice that the differential equation \eqref{lambda-eq} can be written as the planar system
\begin{equation}\label{HS-lambda}
\begin{cases}
\phantom{-}x' = \sqrt\lambda\, y,
\\
-y' = \sqrt\lambda \, g(t,x).
\end{cases}
\end{equation}
This is a Hamiltonian system with Hamiltonian function
\begin{equation*}
H_\lambda(t,x,y)= \sqrt\lambda \, \left( \tfrac{1}{2} y^2 + G(t,x) \right),
\end{equation*}
where $G(t,x)$ is such that  $\frac{\partial}{\partial x}G(t,x)= g(t,x)$.

Recalling the modified polar coordinates in \eqref{modpolcor}, the angular velocity is given by
\begin{equation}
\label{angvel}
\dot\theta(t)= \sqrt\lambda \ \frac{y^2(t) + x(t) g(t,x(t))}{x^2(t)+y^2(t)}.
\end{equation}

Exploiting \eqref{g-superlinear}, we deduce that there exists $\overline r\geq 1$
such that
\begin{equation}\label{old-good-sign}
x g(t,x)\geq x^2, \quad \text{for every $(t,x)\in\mathbb{R}^{2}$ with $|x|\geq \overline r$.}
\end{equation}
Then, since $xg(t,x)$ is bounded when $|x|\leq \overline r$, we have the following.

\begin{proposition}\label{theta'pos}
There exist $c\in\mathopen{]}0,1\mathclose{[}$ and $r_0\geq \overline r$ such that
\begin{equation*}
\frac{y^2 + x g(t,x)}{x^2+y^2}\geq  c,
\quad \text{for every $(t,x,y)\in\mathbb{R}\times\mathbb{R}^{2}$ with $x^2+y^2\geq r_0^2$.}
\end{equation*}
\end{proposition}

The remaining part of the proof is divided into two steps.
In Step~1 we construct some guiding curves in the phase plane, independently of $\lambda$, which will control the behaviour of the solutions while rotating around the origin.
In Step~2 we suitably modify the original system so to introduce an equilibrium at the origin. Then, the Poincar\'e--Birkhoff theorem for the successor map applies to this modified system, providing the multiplicity of periodic solutions. The proof is finally completed by checking that these periodic solutions are indeed solutions of the original system.

\smallskip

\paragraph{Step~1. The construction of some guiding curves.}

We define the continuous functions $g_1,g_2 \colon \mathbb{R} \to \mathbb{R}$ as
\begin{equation}\label{who-are-g}
g_1(x):= \min_{t\in[0,T]} g(t,x)-1 < g(t,x) < g_2(x):= \max_{t\in[0,T]} g(t,x)+1 ,
\end{equation}
and their primitives $G_i(x):=\int_0^x g_i(s) \,\mathrm{d}s$, $i=1,2$. We set
\begin{equation*}
H_i(x,y)= \tfrac12 y^2 + G_i(x) ,
\end{equation*}
and note that
\begin{equation}\label{ordering}
\begin{aligned}
G_1(x) > G_2(x), &\quad\text{for every $x\in\mathopen{]}-\infty,0\mathclose{[}$,}
\\
G_1(x) < G_2(x), &\quad\text{for every $x\in\mathopen{]}0,+\infty\mathclose{[}$.}
\end{aligned}
\end{equation}
Furthermore, since
\begin{equation*}
\lim_{x\to\pm\infty} \frac{G_i(x)}{x^2} =\lim_{x\to\pm\infty} \frac{g_i(x)}{x} =+\infty,
\end{equation*}
there exists a positive constant $E_0$ such that, for every $E\geq E_0$ and $i=1,2$, the sublevel sets $\{H_i \leq E\}$ are star-shaped with respect to the origin and
\begin{equation}\label{outside}
\{(x,y) \colon H_i(x,y)=E\}\subseteq \{(x,y) \colon x^2+y^2> r_0^2+1\}.
\end{equation}
Moreover, the set $\{H_i=E\}$ intersects every semi-axis at a unique point.

\smallskip

\paragraph{Step~1a. Entering solution guiding curve.}
Let us construct a guiding curve $\bighat\gamma$ in the phase plane having the shape of a spiral which performs $m+1$ {\em counterclockwise} rotations around the origin passing through some points $\bighat P_j=(0,\bighat y_j)$, with 
\begin{equation*}
0<\bighat y_1<\dots<\bighat y_j< \bighat y_{j+1} < \dots < \bighat y_{m+2}.
\end{equation*}

As a starting point we can consider any 
\begin{equation}\label{start-hat}
\bighat P_1=(0,\bighat y_1), \quad \text{with $\bighat y_1>\sqrt{2E_0}$,}
\end{equation}
so that $H_2(0,\bighat y_1)>E_0$. 
Once the value $\bighat y_j$ is chosen, we can define the $j$-th rotation of the spiral by gluing three different subsets $\bighat \AA_{j,1}$, $\bighat \AA_{j,2}$, and $\bighat \AA_{j,3}$ (see Figure~\ref{hatspiral}a).
The first part is
\begin{equation*}
\bighat \AA_{j,1} = \bigl{\{} H_2=\tfrac12 \bighat y_j^2 \bigr{\}}\cap\bigl{\{}
x\leq 0, \; y\geq 0\bigr{\}},
\end{equation*}
and links the point $\bighat P_j$ to the point $\bighat Q_j=(\bighat x_j,0)$, with $\bighat x_j<0$ satisfying $G_2(\bighat x_j)=\frac12 \bighat y_j^2$. The second part is
\begin{equation*}
\bighat \AA_{j,2} = \{H_1=H_1(0,\bighat x_j)\}\cap\{ y\leq 0\},
\end{equation*}
and links the point $\bighat Q_j$ to the point $\bighat Q_j'=(\bighat x_j',0)$, with $\bighat x_j'>0$ satisfying $G_1(\bighat x_j')=G_1(\bighat x_j)$. Finally, 
\begin{equation*}
\bighat \AA_{j,3} = \{H_2=H_2(\bighat x_j',0)\}\cap\{ x\geq 0 , y\geq 0\}
\end{equation*}
links the point $\bighat Q_j'$ to the point $\bighat P_{j+1}=(0,\bighat y_{j+1})$, with $\bighat y_{j+1}>0$ satisfying $\frac12 \bighat y_{j+1}^2=G_2(\bighat x_j')$.

By construction, recalling \eqref{ordering}, we have
\begin{equation*}
\tfrac12 \bighat y_j^2 = G_2(\bighat x_j)<G_1(\bighat x_j)=G_1(\bighat x_j')< G_2(\bighat x_j') = \tfrac12 \bighat y_{j+1}^2,
\end{equation*}
hence $\bighat y_{j+1}>\bighat y_j$.
Introducing the segment
\begin{equation*}
\bighat \LL_j = \{(0,y) \colon \bighat y_j \leq y< \bighat y_{j+1} \},
\end{equation*}
the set $\bighat \CC_j = \bighat \AA_{j,1}\cup\bighat \AA_{j,2}\cup\bighat \AA_{j,3}\cup\bighat \LL_j$
is a Jordan curve separating an interior open bounded region $\bighat \II_j$ and an exterior open unbounded region $\bighat \EE_j$.
Note that, since we have assumed \eqref{start-hat}, by \eqref{outside} we have
\begin{equation}\label{inclusion}
\{x^2+y^2\leq r_0^2+1\} \subseteq \bighat \II_1 \subseteq \bighat \II_2 \subseteq \ldots \subseteq \bighat \II_{m+1}.
\end{equation}

\begin{figure}[t]
\centering
a)
\begin{tikzpicture}[x=3mm,y=3mm]
\fill [color=interno,line width=0.8pt] plot [smooth] coordinates {(0,3) (-2.1,2.6) (-3.2,1.8)  (-4,0) (-3.5,-2) (-2,-4)(0,-5) (2,-5) (4.4,-3.9) (6,0) (5,4) (3,6) (0.1,7) (0,7) (0,6.9) (0,3)};
\draw [color=gray,line width=0.8pt] plot [smooth] coordinates {(0,3) (-2.1,2.6) (-3.2,1.8)  (-4,0) (-3.5,-2) (-2,-4)(0,-5) (2,-5) (4.4,-3.9) (6,0) (5,4) (3,6) (0.1,7) (0,7)};
\filldraw[fill=palla,line width=0.8pt] (0,0) circle (1.5);
\draw[->,line width=0.8pt] (-8,0) -- (9,0);
\draw (9.6,0) node {$x$};
\draw[->,line width=0.8pt] (0,-8) -- (0,10.5);
\draw (0,11.1) node {$y$};
\draw (1,3) node {$\bighat P_j$};
\draw (-5,-1) node {$\bighat Q_j$};
\draw (7,-1) node {$\bighat Q_j'$};
\draw (1.5,8) node {$\bighat P_{j+1}$};
\draw (-4.3,3) node {$\bighat \AA_{j,1}$};
\draw (1.5,-6) node {$\bighat \AA_{j,2}$};
\draw (5,6) node {$\bighat \AA_{j,3}$};
\draw (-1,6) node {$\bighat \LL_j$};
\fill (0,3) circle (1.2pt);
\fill (-4,0) circle (1.2pt);
\fill (6,0) circle (1.2pt);
\fill (0,7) circle (1.2pt);
\draw[->,color=black,line width=0.8pt] (4.5,3.5) -- (5.5,4.5);
\draw[->,color=black,line width=0.8pt] (-1.5,-3.5) -- (-2.5,-4.5);
\draw[->,color=black,line width=0.8pt] (3.7,-3.7) -- (4.7,-4.7);
\draw[->,color=black,line width=0.8pt] (-1.5,2.2) -- (-2.5,3.2);
\draw[->,color=black,line width=0.8pt] (-0.3,5) -- (1,5);
\end{tikzpicture}
\qquad
b)
\begin{tikzpicture}[x=3mm,y=3mm]
\fill [color=interno,line width=0.8pt] plot [smooth] coordinates {(0,5.6) (-3,5) (-5.5,2) (-6,0) (-5,-3) (-2.5,-5.5)
(0,-6.3) (6,-4.5) (8,0) (6,5.5) (2,8) (0,8.4)};
\fill [color=white,line width=0.8pt] plot [smooth] coordinates 
{(0,3) (-1,2.8) (-2.8,1.7) (-3.3,0) (-2.8,-1.9) (-1.5,-3.3) (0,-3.6) (2.5,-3.3) (4,-2.5) (5.3,0) (4.3,3.3) (2,5) (0,5.6)};
\draw [color=gray,line width=0.8pt] plot [smooth] coordinates {(0,3) (-1,2.8) (-2.8,1.7) (-3.3,0) (-2.8,-1.9) (-1.5,-3.3) (0,-3.6) (2.5,-3.3) (4,-2.5) (5.3,0) (4.3,3.3) (2,5) (0,5.6) (-3,5) (-5.5,2) (-6,0) (-5,-3) (-2.5,-5.5) (0,-6.3) (6,-4.5) (8,0) (6,5.5) (2,8) (0,8.4)};
\draw [dashed, color=FFS-2,line width=0.8pt] plot [smooth] coordinates {(2,1) (2.5,3) (0,4.3) (-2,3.8) (-3.8,2) (-4.6,0)  (-3.5,-3) (-2,-4.5) (0,-5) (4,-4) (6.6,0) (5,4.3) (0,7) (-2,9)};
\draw[->,color=FFS-2,line width=0.8pt] (4.65,4.6) -- (4.9,4.4);
\draw[->,color=FFS-2,line width=0.8pt] (4.6,-3.44) -- (4.41,-3.64);
\draw[->,color=FFS-2,line width=0.8pt] (-3.46,-3.05) -- (-3.6,-2.83);
\draw[->,color=FFS-2,line width=0.8pt] (-3.27,2.65) -- (-3.05,2.9);
\filldraw[fill=palla, line width=0.8pt] (0,0) circle (1.5);
\draw[->,line width=0.8pt] (-8,0) -- (9,0);
\draw (9.6,0) node {$x$};
\draw[->,line width=0.8pt] (0,-8) -- (0,10.5);
\draw (0,11.1) node {$y$};
\draw (3,1) node {\small $s_1$};
\draw (-5,6) node {\small $\tau_1=\sigma_1$};
\draw (6,8) node {\small $\tau_0=\sigma_2$};
\draw (-3,9) node {\small $s_0$};
\draw (2,-2) node {$\bighat \II_1$};
\draw (7,6) node {$\bighat \EE_2$};
\draw[color=black,line width=0.5pt] (4,8) -- (0,7);
\draw[color=black,line width=0.5pt] (-3,6) -- (0,4.3);
\fill (2,1) circle (1.2pt);
\fill (0,4.3) circle (1.2pt);
\fill (0,7) circle (1.2pt);
\fill (-2,9) circle (1.2pt);
\end{tikzpicture}
\caption{a) The set $\bighat \CC_j$ delimiting the interior regions $\bighat\II_j$ and the exterior regions $\bighat\EE_j$. The arrows suggest the direction of the vector field associated with \eqref{HS-lambda}, see Lemma~\ref{lem-xx}. b) A sketch of Proposition~\ref{prop-xx} in the case $m=1$.}
\label{hatspiral}
\end{figure}

\begin{lemma}\label{lem-xx}
Let $z=(x,y)$ be a solution of \eqref{HS-lambda} such that, for a certain $t_0$, we have 
$z(t)\in \bighat \EE_j$
in a left neighborhood of $t_0$ and
$z(t)\in \bighat \II_j$
in a right neighborhood of $t_0$. Then $z(t_0)\in \bighat \LL_j$.
\end{lemma}

\begin{proof}
We can compute the derivative of the energy $H_i$ along the trajectories of system \eqref{HS-lambda}
\begin{equation}\label{H-der}
\dfrac{\mathrm{d}}{\mathrm{d}t} H_i(x(t),y(t)) = \sqrt\lambda y(t) \big( g_i(x) - g(t,x) \big),
\end{equation}
so that
\begin{equation*}
\dfrac{\mathrm{d}}{\mathrm{d}t} H_1(x(t),y(t)) > 0, \; \text{ if $y(t)<0$,} \qquad
\dfrac{\mathrm{d}}{\mathrm{d}t} H_2(x(t),y(t)) > 0, \; \text{ if $y(t)>0$.}
\end{equation*}
Then, by the properties of the sets $\bighat \AA_{j,\kappa}$ we easily conclude.
\end{proof}

\begin{proposition}\label{prop-xx}
Let $z=(x,y)$ be a solution of \eqref{HS-lambda} and $s_0<s_1$ be such that 
$z(s_0)\in \bighat \EE_{m+1}$
and 
$z(s_1)\in \bighat \II_1$.
Then, there are $\tau_0$ and $\tau_1$, with $s_0\leq \tau_0<\tau_1\leq s_1$, with the following property:
$z(\tau_0)\in \bighat \LL_{m+1}$, 
$z(\tau_1)\in \bighat \LL_1$,
and $\Rot(z,[\tau_0,\tau_1])\geq m$.
\end{proposition}

\begin{proof}
From Lemma~\ref{lem-xx}, we deduce the existence of some instants $\sigma_j\in\mathopen{]}s_0,s_1\mathclose{[}$, with $j\in\{1,\dots,m+1\}$, and $\sigma_{j+1}<\sigma_j$ such that $z(\sigma_j)\in \bighat\LL_j$.
Then, recalling Proposition~\ref{theta'pos} we conclude.
\end{proof}

\smallskip

\paragraph{Step~1b. Exiting solution guiding curve.}
We now construct a spiral $\bigcheck\gamma$  which performs $m+1$ {\em clockwise} rotations around the origin passing through some points $\bigcheck P_j=(0,\bigcheck y_j)$, with
\begin{equation*}
0<\bigcheck y_1<\dots<\bigcheck y_j< \bigcheck y_{j+1} < \dots < \bigcheck y_{m+2}.
\end{equation*}

As a starting point, we can consider any 
\begin{equation}\label{start-check}
\bigcheck P_1=(0,\bigcheck y_1), \quad \text{with }\bigcheck y_1>\sqrt{2E_0},
\end{equation}
so that $H_2(0,\bigcheck y_1)>E_0$. 
Once the value $\bigcheck y_j$ is chosen, we can define the $j$-th rotation of the spiral by gluing three different subsets $\bigcheck \AA_{j,1}$, $\bigcheck \AA_{j,2}$, and $\bigcheck \AA_{j,3}$  (see Figure~\ref{checkspiral}a).
The first one
\begin{equation*}
\bigcheck \AA_{j,1} = \{H_1=\tfrac12 \bigcheck y_j^2\}\cap\{x\geq 0 , y\geq 0\}
\end{equation*}
links the point $\bigcheck P_j$ to $\bigcheck Q_j=(\bigcheck x_j,0)$, with $\bigcheck x_j>0$ satisfying $G_1(\bigcheck x_j)=\frac12 \bigcheck y_j^2$. The second one is
\begin{equation*}
\bigcheck \AA_{j,2} = \{H_2=H_2(0,\bigcheck x_j)\}\cap\{y\leq 0\},
\end{equation*}
and links $\bigcheck Q_j$ to  $\bigcheck Q_j'=(\bigcheck x_j',0)$, with $\bigcheck x_j'<0$ satisfying $G_2(\bigcheck x_j')=G_2(\bigcheck x_j)$. Finally, 
\begin{equation*}
\bigcheck \AA_{j,3} = \{H_1=H_1(\bigcheck x_j',0)\}\cap\{x\leq 0 , y\geq 0\}
\end{equation*}
links the point $\bigcheck Q_j'$ to  $\bigcheck P_{j+1}=(0,\bigcheck y_{j+1})$, with $\bigcheck y_{j+1}>0$ satisfying $\frac12 \bigcheck y_{j+1}^2=G_1(\bigcheck x_j')$.

By construction, recalling \eqref{ordering}, we have 
\begin{equation*}
\tfrac12 \bigcheck y_j^2 = G_1(\bigcheck x_j)<G_2(\bigcheck x_j)=G_2(\bigcheck x_j')< G_1(\bigcheck x_j') = \tfrac12 \bigcheck y_{j+1}^2,
\end{equation*}
hence $\bigcheck y_{j+1}>\bigcheck y_j$.
Introducing the segment
\begin{equation*}
\bigcheck \LL_j = \{(0,y) \colon \bigcheck y_j < y\leq \bigcheck y_{j+1} \},
\end{equation*}
the set $\bigcheck \CC_j = \bigcheck \AA_{j,1}\cup\bigcheck \AA_{j,2}\cup\bigcheck \AA_{j,3}\cup\bigcheck \LL_j$
is a Jordan curve separating an interior open bounded region  $\bigcheck \II_j$  and an  exterior open unbounded region $\bigcheck \EE_j$.
Note that, since we have assumed \eqref{start-check}, by \eqref{outside} we have
\begin{equation*}
\{x^2+y^2\leq r_0^2+1\} \subseteq \bigcheck \II_1 \subseteq \bigcheck \II_2 \subseteq \ldots \subseteq \bigcheck \II_{m+1}.
\end{equation*}

\begin{figure}[t]
\centering
a)
\begin{tikzpicture}[x=3mm,y=3mm]
\fill [color=interno,line width=0.8pt] plot [smooth] coordinates {(0,3) (2.1,2.6) (3.2,1.8)  (4,0) (3.5,-2) (2,-4)(0,-5) (-2,-5) (-4.4,-3.9) (-6,0) (-5,4) (-3,6) (-0.1,7) (0,7) (0,6.9) (0,3)};
\draw [color=gray,line width=0.8pt] plot [smooth] coordinates {(0,3) (2.1,2.6) (3.2,1.8)  (4,0) (3.5,-2) (2,-4)(0,-5) (-2,-5) (-4.4,-3.9) (-6,0) (-5,4) (-3,6) (-0.1,7) (0,7)};
\filldraw[fill=palla,line width=0.8pt] (0,0) circle (1.5);
\draw[->,line width=0.8pt] (-9,0) -- (8,0);
\draw (8.6,0) node {$x$};
\draw[->,line width=0.8pt] (0,-8) -- (0,10.5);
\draw (0,11.1) node {$y$};
\draw (-1,3) node {$\bigcheck P_j$};
\draw (5,-1) node {$\bigcheck Q_j$};
\draw (-7,-1) node {$\bigcheck Q_j'$};
\draw (1.5,8) node {$\bigcheck P_{j+1}$};
\draw (4.3,3) node {$\bigcheck \AA_{j,1}$};
\draw (-1.5,-6) node {$\bigcheck \AA_{j,2}$};
\draw (-5,6) node {$\bigcheck \AA_{j,3}$};
\draw (1,6) node {$\bigcheck \LL_j$};
\fill (0,3) circle (1.2pt);
\fill (4,0) circle (1.2pt);
\fill (-6,0) circle (1.2pt);
\fill (0,7) circle (1.2pt);
\draw[->,color=black,line width=0.8pt] (-5.5,4.5) -- (-4.5,3.5);
\draw[->,color=black,line width=0.8pt] (2.5,-4.5) -- (1.5,-3.5);
\draw[->,color=black,line width=0.8pt] (-4.7,-4.7) -- (-3.7,-3.7);
\draw[->,color=black,line width=0.8pt] (2.3,3) -- (1.3,2);
\draw[->,color=black,line width=0.8pt] (-1,5) -- (0.5,5);
\end{tikzpicture}
\qquad
b)
\begin{tikzpicture}[x=3mm,y=3mm]
\fill [color=interno,line width=0.8pt] plot [smooth] coordinates {(0,5.6) (3,5) (5.5,2) (6,0) (5,-3) (2.5,-5.5) (0,-6.3) (-6,-4.5) (-8,0) (-6,5.5) (-2,8) (0,8.4)};
\fill [color=white,line width=0.8pt] plot [smooth] coordinates {(0,3) (1,2.8) (2.8,1.7) (3.3,0) (2.8,-1.9) (1.5,-3.3) (0,-3.6) (-2.5,-3.3) (-4,-2.5) (-5.3,0) (-4.3,3.3) (-2,5) (0,5.6)};
\draw [color=gray,line width=0.8pt]plot [smooth] coordinates 
{(0,3) (1,2.8) (2.8,1.7) (3.3,0) (2.8,-1.9) (1.5,-3.3) (0,-3.6) (-2.5,-3.3) (-4,-2.5) (-5.3,0) (-4.3,3.3) (-2,5) (0,5.6) (3,5) (5.5,2) (6,0) (5,-3) (2.5,-5.5) (0,-6.3) (-6,-4.5) (-8,0) (-6,5.5) (-2,8) (0,8.4)};
\draw [dashed, color=FFS-2,line width=0.8pt] plot [smooth] coordinates 
{(-2,1) (-2.5,3) (0,4.3) (2,3.8) (3.8,2) (4.6,0)  (3.5,-3) (2,-4.5) (0,-5) (-4,-4) (-6.6,0) (-5,4.3) (0,7) (2,9)};
\draw[->,color=FFS-2,line width=0.8pt] (-4.3,4.85) -- (-4,5);
\draw[->,color=FFS-2,line width=0.8pt] (3.05,2.9) -- (3.28,2.65);
\draw[->,color=FFS-2,line width=0.8pt] (3.6,-2.8) -- (3.45,-3.1);
\draw[->,color=FFS-2,line width=0.8pt] (-4.9,-3.1) -- (-5.05,-2.9);
\filldraw[fill=palla, line width=0.8pt] (0,0) circle (1.5);
\draw[->,line width=0.8pt] (-9,0) -- (8,0);
\draw (8.6,0) node {$x$};
\draw[->,line width=0.8pt] (0,-8) -- (0,10.5);
\draw (0,11.1) node {$y$};
\draw (-3,1) node {\small $s_1$};
\draw (4,6) node {\small $\tau_0=\sigma_1$};
\draw (-6,8) node {\small $\tau_1=\sigma_2$};
\draw (3,9) node {\small $s_0$};
\draw (-2,-2) node {$\bigcheck \II_1$};
\draw (-7,6) node {$\bigcheck \EE_2$};
\draw[color=black,line width=0.5pt] (0,7) -- (-4,8) ;
\draw[color=black,line width=0.5pt] (0,4.3) -- (2,6);
\fill (-2,1) circle (1.2pt);
\fill (0,4.3) circle (1.2pt);
\fill (0,7) circle (1.2pt);
\fill (2,9) circle (1.2pt);
\end{tikzpicture}
\caption{a) The set $\bigcheck \CC_j$ delimiting the interior regions $\bigcheck\II_j$ and the exterior regions $\bigcheck\EE_j$. The arrows suggest the direction of the vector field associated with \eqref{HS-lambda}, see Lemma~\ref{lem-xx2}. b) A sketch of Proposition~\ref{prop-xx2} in the case $m=1$.}
\label{checkspiral}
\end{figure}

\begin{lemma}\label{lem-xx2}
Let $z=(x,y)$ be a solution of \eqref{HS-lambda} such that, for a certain $t_0$, we have  
$z(t)\in \bigcheck \II_j$ 
in a left neighborhood of $t_0$ and $z(t)\in \bigcheck \EE_j$
in a right neighborhood of $t_0$. Then $z(t_0)\in \bigcheck \LL_j$.
\end{lemma}

\begin{proof}
From \eqref{H-der} we get
\begin{equation*}
\dfrac{\mathrm{d}}{\mathrm{d}t}H_1(x(t),y(t)) < 0, \; \text{ if $y(t)>0$,} \qquad
\dfrac{\mathrm{d}}{\mathrm{d}t}H_2(x(t),y(t)) < 0, \; \text{ if $y(t)<0$.}
\end{equation*}
Then, by the properties of the sets $\bigcheck \AA_{j,\kappa}$ we easily conclude.
\end{proof}

\begin{proposition}\label{prop-xx2}
Let $z=(x,y)$ be a solution of \eqref{HS-lambda} and $s_0<s_1$ be such that $z(s_0)\in \bigcheck \II_{1}$ and $z(s_1)\in \bigcheck \EE_{m+1}$.
Then, there are $\tau_0$ and $\tau_1$, with $s_0\leq \tau_0<\tau_1\leq s_1$, with the following property: 
$z(\tau_0)\in \bigcheck \LL_{1}$, $z(\tau_1)\in \bigcheck \LL_{m+1}$,
and 
$\Rot(z,[\tau_0,\tau_1])\geq m$.
\end{proposition}

\begin{proof}
    From Lemma~\ref{lem-xx2}, we deduce the existence of some instants $\sigma_j\in\mathopen{]}s_0,s_1\mathclose{[}$, 
    with $j\in\{1,\dots,m+1\}$, and $\sigma_{j+1}>\sigma_j$ such that 
    $z(\sigma_j)\in \bigcheck\LL_j$.
    Then, recalling Proposition~\ref{theta'pos} we conclude.
\end{proof}

\smallskip

\paragraph{Step~2. The modified problem.}
Recalling the definition of $r_0$ in Proposition~\ref{theta'pos}, we set $\sigma_0:=r_0^2+\frac12 <\sigma_1:=r_0^2+1$.
Let $\varphi\colon\mathopen{[}0,+\infty\mathclose{[}\to\mathopen{[}0,1\mathclose{]}$ be a function of class $\mathcal{C}^{\infty}$ such that $\varphi'\geq0$ and
\begin{equation*}
\varphi(\sigma)=
\begin{cases}
\, 0, & \text{if $\sigma \leq \sigma_0$,}
\\
\, 1, & \text{if $\sigma \geq \sigma_1$.}
\end{cases}
\end{equation*}
Without loss of generality, we can assume that 
\begin{equation}
    \label{choiceG}
    G(t,x) - \tfrac{1}{2} cx^{2}  \geq 0, \quad\text{for every $(t,x)\in[0,T]\times \R$,}
\end{equation}
where $c$ is the constant introduced in Proposition~\ref{theta'pos}.

We consider the modified Hamiltonian function
\begin{equation*}
\widetilde{\mathcal{H}}(t,x,y) = \sqrt{\lambda} \left[ \tfrac{1}{2} (cx^{2} + y^{2})+ \left( G(t,x) - \tfrac{1}{2} cx^{2} \right) \varphi (x^{2}+y^{2})\right] 
\end{equation*}
and the associated Hamiltonian system
\begin{equation}\label{HS-modif}
\begin{cases}
\, \phantom{-}\dot{x} = \sqrt{\lambda}\, y \left[ 1 + 2 \bigl{(} G(t,x) -  \tfrac{1}{2} c x^{2} \bigr{)} \varphi'(x^{2}+y^{2})  \right],
\\
\, -\dot{y} =  \sqrt{\lambda} \left[ cx + \left( g(t,x)-cx \right)\varphi(x^{2}+y^{2}) \right.
\\
\phantom{-\dot{y} =  \sqrt{\lambda} \left[ \right.} \left.
+2x\left( G(t,x) - \tfrac{1}{2}c x^{2} \right) \varphi'(x^{2}+y^{2}) 
\right].
\end{cases}
\end{equation}
We notice that \eqref{HS-modif} corresponds to \eqref{HS-lambda} whenever $x^{2}+y^{2}\geq \sigma_1$ and reduces to the autonomous linear system
\begin{equation*}
\begin{cases}
\, \phantom{-}\dot{x} = \sqrt{\lambda}\, y,
\\
\, -\dot{y} = \sqrt{\lambda}\,cx,
\end{cases}
\end{equation*}
whenever $x^{2}+y^{2} \leq \sigma_0$.

In order to define the successor map associated with \eqref{HS-modif}, we are going to verify that conditions \ref{hp-A-1}--\ref{hp-A-5} are satisfied. We name the functions in \eqref{HS-modif} as
\begin{align*}
f_\lambda(t,x,y)&=
\sqrt{\lambda}\, y \left[ 1 + 2 \bigl{(} G(t,x) -  \tfrac{1}{2} c x^{2} \bigr{)} \varphi'(x^{2}+y^{2})  \right],\\
g_\lambda(t,x,y)&=
\sqrt{\lambda} \left[ cx + \left( g(t,x)-cx \right)\varphi(x^{2}+y^{2}) 
+2x\left( G(t,x) - \tfrac{1}{2}c x^{2} \right) \varphi'(x^{2}+y^{2}) 
\right].
\end{align*}

To show that both \ref{hp-A-1} and \ref{hp-A-2} hold true we prove \eqref{gxfy} for every $(t,x,y)\in[0,T]\times \R^2$.
From the properties of $\varphi$ and Proposition~\ref{theta'pos}, we deduce that \eqref{gxfy} holds when $x^2+y^2\in\mathopen{[}0,\sigma_0\mathclose{]}\cup\mathopen{[}\sigma_1,+\infty\mathclose{[}$. 
It is sufficient to prove it for $x^{2}+y^{2}\in\mathopen{]}\sigma_0,\sigma_1\mathclose{[}$. Using \eqref{choiceG} and the fact that $\varphi(\sigma)\in[0,1]$ for every $\sigma$, we have
\begin{align*}
&g_\lambda(t,x,y)x+f_\lambda(t,x,y)y =
\\
&=\sqrt\lambda \left[
(g(t,x)x+y^{2})\varphi(x^{2}+y^{2})
+(c x^{2} +y^{2}) (1-\varphi(x^{2}+y^{2}))
\right.
\\
&\left. \qquad\quad  + 2 \bigl{(} G(t,x) - \tfrac{1}{2} c x^{2} \bigr{)} \varphi'(x^{2}+y^{2}) (x^{2}+y^{2}) \right]
\\
&\geq \sqrt\lambda\,c (x^2+y^2).
\end{align*}
Concerning condition \ref{hp-A-3}, it is easily verified noticing that $\varphi'=0$ outside a compact set. Finally, 
\ref{hp-A-4} and \ref{hp-A-5} are satisfied
choosing $\delta$ sufficiently small and $D=\sqrt{\sigma_1}$, recalling \eqref{old-good-sign},
since
\begin{equation*}
    y f_\lambda(t,x,y) \geq \sqrt\lambda y^2,\quad 
    x g_\lambda(t,x,y) \geq \sqrt\lambda \bigl{[}cx^2 (1-\varphi(x^2+y^2))+xg(t,x) \varphi(x^2+y^2) \bigr{]}
\end{equation*}
hold everywhere.
Having verified \ref{hp-A-1}--\ref{hp-A-5}, Theorem~\ref{successor_thm} guarantees that the successor map $\SuccMap_\lambda$ associated with system \eqref{HS-modif} and its iterates $\SuccMap^m_\lambda \colon\R\times \mathopen{]}0,+\infty\mathclose{[}\to\R\times \mathopen{]}0,+\infty\mathclose{[}$ are well defined. We will use the notation 
\begin{equation*}
\SuccMap^m_\lambda(t_0,y_0)=(\TT^m_\lambda(t_0,y_0),\YY^m_\lambda(t_0,y_0))
\end{equation*}
for their components. 

\begin{proposition}\label{apb}
For any integer $m\geq 1$, there exists $Y_m>0$ such that 
\begin{equation}\label{Sm}
\SuccMap^m_\lambda(\R\times\mathopen{[}Y_m,+\infty\mathclose{[}) \subseteq \R\times\mathopen{[}\sigma_1,+\infty\mathclose{[}, \quad \text{for every $\lambda>0$.}   
\end{equation}
Moreover, every solution $(x,y)$ of \eqref{HS-modif}, with $x(t_0)=0$, $y(t_0)=y_0\geq Y_m$,
satisfies 
\begin{equation}\label{large-enough}
x(t)^2+y(t)^2>\sigma_1, \quad \text{for every $t\in[t_0,\TT^m_\lambda(t_0,y_0)]$,} 
\end{equation}
hence it
is also a solution of \eqref{HS-lambda} there.
\end{proposition}

\begin{proof}
We construct a guiding curve $\bighat\gamma$ performing $m+1$ rotations around the origin as explained in Step~1a by choosing $\bighat y_1>\sigma_1$. We set $Y_m=\bighat y_{m+2}$.
Then, from Proposition~\ref{prop-xx}, we get \eqref{Sm}.
Recalling \eqref{outside} and \eqref{inclusion}, by construction we get
$\{x^2+y^2 \leq \sigma_1\}\subseteq \bighat\II_1 \subseteq \bighat\EE_{m+1}$ and the second assertion holds, as well. 
\end{proof}

\begin{proposition}\label{PB-int}
For any pair $(m,k)$ of positive integers there exists $\lambda_{m,k}>0$ such that
    \begin{equation*}
    \TT^m_\lambda(t_0,Y_m)-t_0>kT, 
    \quad\text{for every $t_0\in\mathbb{R}$ and $\lambda\in\mathopen{]}0,\lambda_{m,k}\mathclose{[}$.}
    \end{equation*}
\end{proposition}

\begin{proof}
We first construct the guiding curve $\bigcheck\gamma$ performing $m+1$ rotations around the origin as explained in Step~1b, by choosing $\bigcheck y_1>Y_m$. 
Then, we define the compact set $\mathcal{K}_m$ as the closure of ${\bigcheck\II}_{m+1}\cap \{x^2+y^2 \geq\sigma_1\}$.
From Proposition~\ref{prop-xx2} we deduce that every solution $(x,y)$ of \eqref{HS-modif}, with $x(t_0)=0$, $y(t_0)= Y_m$, satisfies $(x(t),y(t))\in\mathcal{K}_m$ for every $t\in[t_0, \TT^m_\lambda(t_0,Y_m)]$.
So, introducing modified polar coordinates as in \eqref{modpolcor}, recalling the expression for the angular velocity \eqref{angvel}, we deduce the existence of $\Theta_m>0$ such that that
\begin{equation*}
|\theta'(t)| \leq \sqrt{\lambda} \,\Theta_m,
\quad \text{for every $t\in[t_0, \TT^m_\lambda(t_0,Y_m)]$.}
\end{equation*}
Hence, $\TT^m_\lambda(t_0,Y_m) - t_0 \geq \frac{2\pi m}{\sqrt{\lambda} \Theta_m}$, and we conclude by setting $\lambda_{m,k}=\bigl(\frac{2\pi m}{kT\Theta_m}\bigr)^{\!2}$.
\end{proof}

\begin{proposition}\label{PB-est}
For any pair $(m,k)$ of positive integers and any $\lambda>0$ there exists $Z_{m,k,\lambda}>Y_m$ such that
\begin{equation*}
\TT^m_\lambda(t_0,y_0)-t_0<kT, 
\quad\text{for every $t_0\in\mathbb{R}$ and $y_0\geq Z_{m,k,\lambda}$.}
\end{equation*}
\end{proposition}

\begin{proof}
By the superlinear assumption \eqref{g-superlinear}, recalling the expression for the angular velocity \eqref{angvel}, for every $M>0$ we can find a radius $R_M>0$ such that any solution $(x,y)$ of \eqref{HS-modif} satisfies
\begin{equation*}
\theta'(t) \geq \frac12 \sqrt\lambda [\cos^2(\theta(t)) + M \sin^2(\theta(t))],
\quad \text{when $x^2(t)+y^2(t)\geq R_M^2$.}
\end{equation*}
Let us consider a solution $z=(x,y)$ of \eqref{HS-modif} such that $\Rot(z,[t_0,t_1])=m$ and $x^2(t)+y^2(t)\geq R_M^2$ for every $t\in[t_0,t_1]$. Then, integration in the above inequality, with the change of variable $s=\tan \theta$, gives us
\begin{align*}
t_1-t_0 &\leq \dfrac{2}{\sqrt\lambda} \int_0^{2\pi m} \dfrac{\mathrm{d}\theta}{\cos^2\theta + M \sin^2\theta}
= \frac{8m}{\sqrt\lambda} \int_0^{\frac{\pi}{2}} \dfrac{\mathrm{d}\theta}{\cos^2\theta + M \sin^2\theta}\\
& = \frac{8m}{\sqrt\lambda}
\int_0^{+\infty} \dfrac{\mathrm{d}s}{1+Ms^2}= \frac{1}{\sqrt M} \, \frac{4m\pi}{\sqrt\lambda} < kT,
\end{align*}
where the last inequality holds for $M$ large enough.
In order to well define the value $Z_{m,k,\lambda}$ we introduce a new spiral $\bighat\gamma$ as in Step~1a, setting $\bighat y_1>R_M$ such that $\{x^2+y^2 \leq R_M^2\}\subseteq \bighat\II_1$. Then, we fix $Z_{m,k,\lambda}>\max\{\bighat y_{m+2}, Y_m\}$. The above reasoning concludes the proof.
\end{proof}

Summing up, we have proved that
for every positive integers $m$ and $k$, there exist $Y_m>0$ and $\lambda_{m,k}>0$, such that for every $\lambda \in \mathopen{]}0,\lambda_{m,k}\mathclose{[}$ there is $Z_{m,k,\lambda}>Y_m$ satisfying
\begin{equation*}
\TT^m_\lambda(t_0,Y_m) -t_0 > kT > \TT^m_\lambda(t_0,Z_{m,k,\lambda})-t_0,
\quad\text{for every $t_0\in\R$.}
\end{equation*}
Hence, applying Theorem~\ref{PB-thm}, for every couple $(m,k)$ of positive integers and every $\lambda\in\mathopen{]}0,\lambda_{m,k}\mathclose{[}$ there are two distinct $kT$-periodic solutions of \eqref{HS-modif}
satisfying 
$\Rot(z, [0,kT])=m$. 
By Proposition~\ref{apb}, such solutions are indeed solutions of \eqref{HS-lambda},  hence the corresponding $x$-components solve \eqref{lambda-eq} and have exactly $2m$ simple zeros in the period interval $\mathopen{[}0,kT\mathclose{[}$. 
The proof of Theorem~\ref{THM_lambda} is thus completed under the assumption that $\partial_t g$ exists and is continuous.

\smallskip

Let us now treat the general case. We first need the following estimate.

\begin{proposition}\label{apb-est}
For any pair $(m,k)$ of positive integers and any $\lambda>0$ there exists $\Sigma_1>0$ such that, for every $t_0\in\R$ and every $y_0\in[Y_m,Z_{m,k,\lambda}]$, all the solutions $z=(x,y)$ of \eqref{HS-modif} with $x(t_0)=0$ and $y(t_0)=y_0$ satisfy
\begin{equation*}
    x^2(t)+y^2(t)< \Sigma_1, \quad \text{for every $t \in [t_0, \TT^m_\lambda(t_0,y_0)]$.}
\end{equation*}
\end{proposition}

\begin{proof}
    We need to construct a new spiral $\bigcheck\gamma$ as in Step~1b performing $m+1$ clockwise rotations around the origin, choosing $\bigcheck y_1> Z_{m,k,\lambda}$. Then, we select $\Sigma_1$ so that $\bigcheck\EE_{m+1} \subseteq \{x^2+y^2<\Sigma_1\}$.
\end{proof}

Now, we consider a sequence of mollifiers $\rho_n \colon\R\to \R$ and define $h_n \colon \R\times \R\to \R$ as
\begin{equation*}
h_n(t,x) = \int_\R g(s,x) \rho_n(t-s)\, \mathrm{d}s = \int_\R g(t-\sigma,x) \rho_n(\sigma)\, \mathrm{d}\sigma,
\end{equation*}
which converges uniformly on  $[0,T]\times [-\sqrt{\Sigma_1},\sqrt{\Sigma_1}]$ to $g(t,x)$.
Notice that these functions $h_n(t,x)$ are $T$-periodic in $t$ and $\partial_t h_n$ exist and are continuous. In particular, for $n$ large enough, \eqref{who-are-g} holds replacing $g(t,x)$ with $h_n(t,x)$, provided that $|x|\leq \sqrt{\Sigma_1}$. Therefore, all the above construction can be replicated for system
\begin{equation}\label{spqr}
\begin{cases}
\phantom{-}x' = \sqrt\lambda\, y,
\\
-y' = \sqrt\lambda \, h_n(t,x),
\end{cases}
\end{equation}
thus finding the same constants in the above propositions, independently of $n$. Given any pair $(m,k)$ of positive integers, we so find two $kT$-periodic solutions for \eqref{spqr}, for every $\lambda \in \mathopen{]}0,\lambda_{m,k}\mathclose{]}$, having rotation number $m$.
By a standard argument, involving the use of the Ascoli--Arzel\`a theorem, we recover the existence of a $kT$-periodic solution for \eqref{HS-lambda} having exactly $2m$ simple zeros in the interval $\mathopen{[}0, kT \mathclose{[}$ as the limit of a subsequence of the $kT$-periodic solutions found for \eqref{spqr}. Notice that the multiplicity might get lost in the limiting process.

The proof of Theorem~\ref{THM_lambda} is thus completed.
\qed

\section*{Acknowledgements}

The first author acknowledges the support of INdAM-GNAMPA project ``Analisi qualitativa di problemi differenziali non lineari'' and PRIN Project 20227HX33Z ``Pattern formation in nonlinear phenomena''.
The second and third authors acknowledge the support of PRIN Project 2022ZXZTN2 ``Nonlinear differential problems with applications to real phenomena''.

The research contained in this paper was carried out within the framework of DEG1 -- Differential Equations Group Of North-East.

\bibliographystyle{elsart-num-sort.bst} 
\bibliography{FFS-biblio}

\end{document}